\documentclass[a4paper]{article}

\usepackage[dvipsnames]{xcolor}
\usepackage[utf8]{inputenc}
\usepackage[T1]{fontenc}

\usepackage{amsmath, amssymb, amsthm}
\usepackage{graphicx}
\usepackage{csquotes}
\usepackage{bbm}
\usepackage{amsfonts}
\usepackage{mathrsfs}
\usepackage{csquotes}
\usepackage{enumerate}
\usepackage{epigraph}
\usepackage{microtype}
\usepackage{thmtools}
\usepackage{thm-restate}
\usepackage{tikz}
\usepackage{hyperref}
\hypersetup{
  colorlinks,
  citecolor=Blue,
  linkcolor=Fuchsia,
  urlcolor=OliveGreen
}
\usepackage[capitalize, nameinlink]{cleveref}

\usepackage[
backend=biber,
isbn=false,
date=year,
giveninits=true,
style=alphabetic]
{biblatex}

\bibliography{references}
\usetikzlibrary{cd}
\usetikzlibrary{calc}
\newcommand{\email}[1]{\href{mailto:#1}{#1}}
\title{Eventual Conjugacy of Free Inert $G$-SFTs}
\author{Jeremias Epperlein\thanks{University of Passau,
  Department of Computer Science and Mathematics,
  94032 Passau, Germany,
\email{jeremias.epperlein@uni-passau.de}}}
\date{\today}

\newtheorem{thm}{Theorem}[section]
\newtheorem{lem}[thm]{Lemma}
\newtheorem{defn}[thm]{Definition}
\newtheorem{prop}[thm]{Proposition}
\newtheorem{exam}[thm]{Example}
\newtheorem{quest}[thm]{Question}
\newtheorem{cor}[thm]{Corollary}
\newtheorem{rem}[thm]{Remark}
   \newtheoremstyle{TheoremNum}
        {\topsep}{\topsep}              
        {\itshape}                      
        {}                              
        {\bfseries}                     
        {.}                             
        { }                             
        {\thmname{#1}\thmnote{ #3}}
    \theoremstyle{TheoremNum}
    \newtheorem{thmn}{Theorem}
    \newtheorem{corn}{Corollary}
\newcommand{\blockmat}[1]{  \left(\begin{array}{@{}c|c@{}}#1\end{array}\right)}

\newcommand{\setsep}{\:|\:}

\newcommand{\bbN}{\mathbb{N}}

\newcommand{\bbZ}{\mathbb{Z}}
\newcommand{\bbZP}{\mathbb{Z}_+}
\newcommand{\bbC}{\mathbb{C}}
\newcommand{\calG}{\mathcal{G}}
\newcommand{\graph}{\calG}
\newcommand{\calA}{\mathcal{A}}
\newcommand{\calE}{\mathcal{E}}
\newcommand{\calS}{\mathcal{S}}
\newcommand{\calR}{\mathcal{R}}

\newcommand{\mat}[1]{\begin{pmatrix}#1\end{pmatrix}}
\newcommand{\init}{\mathrm{i}}
\newcommand{\term}{\mathrm{t}}
\newcommand{\abs}[1]{|#1|}
\newcommand{\ones}{\mathbbm{1}}
\DeclareMathOperator{\EL}{EL}
\DeclareMathOperator{\GL}{GL}
\DeclareMathOperator{\SL}{SL}

\DeclareMathOperator{\Aut}{Aut}

\DeclareMathOperator{\id}{id}

\DeclareMathOperator{\Per}{Per}

\begin{document}
\maketitle
\setlength{\epigraphrule}{0pt}
\epigraph{%
  Pray that it is inert.
}{-- Kolvoord, \textit{The Expanse S3E7}}
\begin{abstract}
  The action of a finite group $G$ on a subshift of finite type $X$ is
  called free,
  if every point has trivial stabilizer, and it is called inert,
  if the induced action on the dimension group of $X$ is trivial.
  
  We show that any two free inert actions of a finite group $G$ on
  an SFT are conjugate by an automorphism
  of any sufficiently high power of the shift space.
  This partially answers a question posed by Fiebig.  
  As a consequence we obtain that every
  two free elements of the stabilized
  automorphism group of a full shift are
  conjugate in this group.
  In addition, we generalize a result of Boyle, Carlsen and Eilers
  concerning the flow equivalence of $G$-SFTs.
\end{abstract}

\section{Introduction and Main Results}
\label{sec:introduction}

On the two-sided full shift with two symbols, i.e. $\{0,1\}^\bbZ$, there is a
natural involution $\tau$ defined by swapping the symbols $0$ and $1$.
It is a free involution in the sense that it does not fix any point.
We can easily generate more such free involutions in the
automorphism group of the full $2$-shift: just conjugate $\tau$ by
an arbitrary, not necessarily free, automorphism.

But are there any other free involutions? Or stated more explicitly:
\begin{quest}
  \label{quest:1}
Are any two fixed-point-free involutions of the full $2$-shift
conjugate by an automorphism of the full $2$-shift?
\end{quest}
This is an old question in symbolic dynamics
that was formulated explicitly for example in \cite[p. 492]{fiebigPeriodicPointsFinite1993b}

Given a finite group $G$, a construction similar to the one above
gives a free action\footnote{An action of a group is called
  \emph{free}, if every point has trivial stabilizer.} of $G$ on the full shift with $|G|$ symbols:
The group $G$ acts on the set $G$ by left multiplication and this
induces a free action on $G^\bbZ$. Again we can obtain other
free $G$-actions on $G^\bbZ$ by conjugation with an arbitrary automorphism.

One hint that the answer to \Cref{quest:1} might be \enquote{yes}
is the fact that it seems very hard to construct free automorphisms
of a given SFT by different means. Marker constructions, see e.g.
\cite[Section 2]{boyleAutomorphismGroupShift1988} or
\cite[Chapter 3]{kitchensSymbolicDynamicsOnesided1998} are the classical source of the
very rich structure of the automorphism group of a subshift of finite
type. But this type of construction rarely produces fixed-point-free
automorphisms.  Additionally we don't know how to decide if there even exists
a free action of $\bbZ/p\bbZ$ on a given SFT by automorphisms.  This
question appears explicitly in \cite[Section
8]{boyleWorkKimRoush2013} as \enquote{a question for another generation.} We will
explain in \Cref{sec:kim-roush} how a partial
result in this direction by Kim and Roush  from
\cite{kimFreeSbActions1997} fits together with
the results in this paper.

While it is hard to construct a free action on a given SFT, there are
methods to construct an abundance of SFTs with free $G$-actions. On
one hand, one can construct them as $G$-extensions of an SFT, see for
example \cite{boyleEquivariantFlowEquivalence2005} as well as
\Cref{sec:matrices-over-integer-group-rings}. On the other hand, one
can start with a not necessarily free $G$-action on an SFT $X$ and remove all
points $x$ from $X$ for which $(gx)_i=x_i$ for some $i \in \bbZ$ and
$g \in G \setminus \{e\}$.  Up to topological conjugacy of the
underlying SFT, every free $G$-action on an SFT can be produced by
both methods, see \Cref{thm:franks-repr} and
\Cref{prop:repr-by-group-ring-matrix}.

There is a simple obstruction for two free $G$-actions
on a subshift to be conjugate as the following example shows.
\begin{exam}[Obstruction to topological conjugacy]\label{ex:000}
  Let $X=\{0,1\}^\bbZ$ be the full $2$-shift and let $Y$ be the subshift
  of $X$ obtained by forbidding the words $000$ and $111$.
  \Cref{fig:graph-000} shows the $3$-higher block representation of $Y$
  as an edge shift.

  \begin{figure}
    \begin{center}
    \begin{tikzpicture}
      \node (01) at (-1,0) {$01$};
      \node (10) at (1,0) {$10$};
      \node (00) at (0,1) {$00$};
      \node (11) at (0,-1) {$11$};
      \draw[-latex] (00) edge[bend right] (01);
      \draw[-latex] (01) edge[bend right] (10);
      \draw[-latex] (10) edge[bend right] (01);
      \draw[-latex] (01) edge[bend right] (11);
      \draw[-latex] (10) edge[bend right] (00);
      \draw[-latex] (11) edge[bend right] (10);
    \end{tikzpicture}
    \caption{The subshift $Y$
      from \Cref{ex:000}.    The automorphism induced by $\tau$ corresponds to a point
    reflection of this graph across its center.
  }
  \label{fig:graph-000}
    \end{center}
  \end{figure}
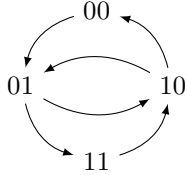
  Let $\tau: X \to X$ be the map induced by the unique free involution
  of $\{0,1\}$.
  The subshift $Y$ is invariant under $\tau$.
  Set $Z:=X \times Y$. Consider the two
  automorphisms of $Z$ given by $\varphi=\id_X \times \tau_{|Y}$
  and $\psi = \tau \times \id_Y$.
  Both automorphisms have no fixed points, since no point is fixed by
  $\tau$. Because automorphisms of subshifts by definition commute with the shift map,
  we also get a shift action on the orbit spaces\footnote{The \emph{orbit
    space}
    is the topological space  obtained by identifying points in every
    orbit, endowed with the
    quotient topology.}
  $Z/\varphi \cong X \times Y/\tau_{|Y}$ and $Z/\psi \cong X/\tau
  \times Y$. Since the actions are free, the orbit
  spaces with the induced shift action are again
  SFTs, see \cite[Theorem 4.1]{silverInvariantFiniteGroup2005}. Notice that this would not necessarily
  be the case if the actions were not free,
  as the shift on the orbit space in this case need not be expansive.
  Now one can easily calculate, see \Cref{exam:extension-and-augmentation}, that $X/\tau \cong X$
  but $Y/\tau_{|Y} \cong U \not\cong Y$ where
  $U$ is the golden mean shift, i.e. the subshift of $X$
  with forbidden word $11$.
  In particular $Z/\varphi$ has two $\sigma$-fixed points,
  while $Z/\psi$ has none. But if $\varphi$
  and $\psi$ were topologically conjugate automorphisms of $Z$, their
  orbit spaces would have to be conjugate shift spaces.

  The deeper reason that makes this example work, is the fact that $\tau$ is inert, i.e. it acts
  trivially on the dimension group of $X$, but the restriction
  $\tau_{|Y}$ is not inert. To see this,
  consider the adjacency matrix of the graph depicted
  in \Cref{fig:graph-000} (with the vertices ordered lexicographically):
  \begin{align*}
   A=\begin{pmatrix}
      0 & 1 & 0 & 0 \\
      0 & 0 & 1 & 1 \\
      1 & 1 & 0 & 0 \\
      0 & 0 & 1 & 0
    \end{pmatrix}
  \end{align*}
  Since $\abs{\det(A)}=1$,
  the dimension group of $Y$ is isomorphic to $\bbZ^4$.
  As we will see in \Cref{sec:finite groups acting freely}, $\tau$ acts on this dimension group 
  by multiplication from the right with the permutation matrix 
  \begin{align*}
    \begin{pmatrix}
      0 & 0 & 0 & 1 \\
      0 & 0 & 1 & 0 \\
      0 & 1 & 0 & 0 \\
      1 & 0 & 0 & 0
    \end{pmatrix}.
  \end{align*}
  Hence $\tau_{|Y}$ is not inert.
  This carries over to the product maps
  $\psi$, which is inert, and $\varphi$, which is not.
  Hence the actions of $\psi$ and $\varphi$ on the dimension
  group are not conjugate and therefore $\psi$ and $\varphi$
  are themselves not conjugate in $\Aut(Z,\sigma)$.
\end{exam}

In light of this example, a natural extension of \Cref{quest:1}
is:
\begin{quest}[Topological conjugacy of inert free actions]
  \label{quest:2}
  Let $X$ be a subshift of finite type and $G$ a finite group. Are any two actions
  of $G$ on $X$ by free inert automorphisms conjugate by an automorphism?
  What about the case that $X$ is a full shift and $G = \bbZ/p\bbZ$?
\end{quest}
It is not totally unreasonable to suspect that this might be true.
For example, two actions of a finite group $G$ acting freely on a
finite set $M$ are always conjugate. Similarly, any two free actions of $G$
on a Cantor space $C$ by homeomorphisms are automatically conjugate
by a self-homeomorphism of $C$.

While we will not be able to answer \Cref{quest:2}, we give a positive
answer to a weaker version.

\begin{thmn}[\ref{thm:main} \textnormal{(Main theorem - dynamical version)}]
  Let $(Y_1,\sigma)$ and $(Y_2,\sigma)$ be
  SFTs such that $(Y_1,\sigma^\ell)$, $(Y_2,\sigma^\ell)$ are
  conjugate
  for all sufficiently large $\ell$.
  Let $\alpha_1$ and $\alpha_2$ be free inert actions of a finite
  group $G$ on $(Y_1,\sigma)$ and $(Y_2,\sigma)$, respectively.
  Then for every sufficiently large $\ell$
  there is a conjugacy $\varphi$ from $(Y_1,\sigma^\ell)$
  to $(Y_2,\sigma^\ell)$ such that $\varphi \circ \alpha_1(g) =
  \alpha_2(g) \circ \varphi$ for all $g \in G$. 
\end{thmn}
So instead of conjugacy of the actions we are only able to obtain eventual conjugacy.

Our main tool will be the integral group ring formalism for $G$-SFTs
developed by Parry, see e.g. \cite{boyleEquivariantFlowEquivalence2005}
and \Cref{sec:matrices-over-integer-group-rings}
together with a characterization of inert $G$-SFTs, 
see \Cref{thm:charact-inert}

In this algebraic formulation our main theorem becomes:
\begin{thmn}[\ref{thm:main-algebraic-version} \textnormal{(Main theorem - algebraic version)}]
  Every pair of inert square matrices over $\bbZP[G]$,
  whose augmentations are shift equivalent over $\bbZP$, is shift equivalent over $\bbZP[G]$ itself.
\end{thmn}

Regarding \Cref{quest:1} we obtain the following corollary.
\begin{corn}[\ref{thm:main-corr-classic-version} \textnormal{(Eventual
    conjugacy for free finite order automorphisms of full shifts)}]
  Let $k,m \in \bbN$.
  Let $(X_k,\sigma)$ be the full $k$-shift and 
  let $\varphi_1,\varphi_2$ be two automorphisms
  of $(X_k,\sigma)$ for which every orbit has size $m$.
  Then there is a homeomorphism $\psi: X_k \to X_k$
  such that  $\psi \circ \varphi_1 = \varphi_2 \circ \psi$
  and $\psi \circ \sigma^\ell = \sigma^\ell
  \circ \psi$ for all sufficiently large $\ell \in \bbN$.
\end{corn}

Maybe the most natural setting for this version of eventual conjugacy
is the stabilized automorphism group of a subshift introduced by
Hartman, Kra and Schmieding in
\cite{hartmanStabilizedAutomorphismGroup2021}.
In this setting our main result reads:
\begin{corn}[\ref{cor:main-corr-stabilized-version} \textnormal{(Conjugacy of free finite order elements in the stabilized automorphism group)}]
  \label{cor:main-corr-stabilized-version}
   Let $k,m \in \bbN$.
   Let $(X_k,\sigma)$ be the full $k$-shift.
   Let $\varphi_1, \varphi_2$ be two elements of the
   stabilized automorphism group of $(X_k,\sigma)$
   for which every orbit has size $m$.
   Then $\varphi_1$ and $\varphi_2$ are conjugate
   in $\Aut^\infty(X_k,\sigma)$.
\end{corn}

We conclude the introduction with
some pointers to related works.
The one-sided version of \Cref{quest:2} in the
case of $\bbZ/k \bbZ$ acting on $(\{1,\dots,k\}^\bbN,\sigma)$ was answered
positively for $k$ prime in \cite[Lemma 3.3
(ii)]{boyleAutomorphismsOnesidedSubshifts1990}
and very recently for $k$ arbitrary in \cite[Theorem 1.2]{bleakConjugacyCertainAutomorphisms2023}.
Notice however, that the automorphism groups of one-sided
SFTs are significantly less rich than
their two-sided counterparts.

One can also consider other weakenings of conjugacy besides eventual conjugacy.
For a classification of $G$-SFTs up to finite equivalence see \cite{adlerFiniteGroupActions1985}.

Boyle, Carlsen and Eilers \cite{boyleFlowEquivalenceGSFTs2020a}, based
on earlier work of Boyle and Sullivan \cite{boyleEquivariantFlowEquivalence2005}
in the irreducible case, classified $G$-SFTs
up to $G$-flow equivalence. As an application
they showed a flow equivalence version of
\Cref{thm:main-corr-classic-version}.
We will show in \Cref{sec:equivariant-flow-equivalence}
how a generalization of this result can be recovered as a corollary of our results.

The paper is organized as follows. We recall a few definitions and notations from
symbolic dynamics in \Cref{sec:notation}. In \Cref{sec:finite groups
  acting freely} we discuss basic results regarding finite groups
acting on subshifts. The integral group ring formalism due to Parry
used to define $G$-extensions of SFTs is explained in
\Cref{sec:matrices-over-integer-group-rings}.
In \Cref{sec:inert-g-sfts} we use the results of the previous
two sections to give a characterization of inert $G$-SFTs,
which in turn leads to the proofs of our main theorems in
\Cref{thm:main-algebraic-version}.
In \Cref{sec:kim-roush} we put a related result
of Kim and Roush into our context and we finish
in \Cref{sec:equivariant-flow-equivalence}
with an application of our main results to
equivariant flow equivalence of $G$-SFTs.

\section{Preliminaries}
\label{sec:notation}

All necessary material concerning subshifts of finite type
can be found in \cite{lindIntroductionSymbolicDynamics2021}.
This section mainly serves to introduce notation.
Our natural numbers $\bbN$ start at $1$
and we denote the non-negative integers by $\bbZP$.
In this paper a \emph{topological dynamical system} $(X,f)$ consists
of a
compact metrizable state space $X$ and a continuous self map $f:X
\to X$. The \emph{full shift} over the \emph{alphabet} $\calS$ is the topological
dynamical system $(\calS^\bbZ,\sigma)$ together with the left
shift $\sigma: \calS^\bbZ \to \calS^\bbZ, \sigma(x)_i = x_{i+1}$, where $\calS$ is endowed
with
the discrete topology and $\calS^\bbZ$ with the product topology.
The elements of $\calS^\bbZ$ are called \emph{configurations}.
A \emph{subshift} is a subsystem
of a full shift. A subshift $X$ is of \emph{finite type}
if there is a finite set of forbidden words, such that
$X$ consists of all configurations not containing any of these words.
An isomorphism between two topological dynamical
system $(X,f)$ and $(Y,g)$
is a homeomorphism $\varphi: X \to Y$ such that
$\varphi \circ f = g \circ \varphi$.
We denote the group of automorphisms of
the topological dynamical system $(X,f)$ by
$\Aut(X,f)$.
Following \cite{hartmanStabilizedAutomorphismGroup2021},
we call $\Aut^{\infty}(X,f) = \bigcup_{k \in \bbN} \Aut(X,f^k)$
the \emph{stabilized automorphism group} of $(X,f)$.

If $(X,\sigma)$ is a subshift over the alphabet $\calS$,
then we can consider $(X,\sigma^k)$ as a
subshift over the alphabet $\calS^k$.
We say that two subshifts $(X,\sigma)$ and $(Y,\sigma)$
are eventually conjugate, iff $(X,\sigma^k)$ and
$(Y,\sigma^k)$ are topologically conjugate for sufficiently
large $k$.

Our \emph{graphs} are tuples $\Gamma=(V,E,\init_\Gamma,\term_\Gamma)$
consisting of a vertex set $V=V(\Gamma)$, an edge set $E=E(\Gamma)$
together with two $\init_\Gamma,\term_\Gamma: E \to V$
mapping edges to their \emph{initial} and \emph{terminal vertices},
respectively.
In this context it will be notational convenient to consider matrices
whose rows and columns are indexed by an arbitrary finite set $M$
not necessarily of the form $\{1,\dots,m\}$.
If $\Gamma=(V,E,\init_\Gamma,\term_\Gamma)$ is a graph,
its adjacency matrix is the $V \times V$ matrix $A$
with \[A_{i,j}=\abs{\{e \in E \setsep \init(e)=i, \term(e)=j\}}.\]

We call a matrix \emph{essential}, if it has no zero rows or columns
and we call a graph essential, if its adjacency matrix is essential.
A non-negative square matrix $A$ is called \emph{irreducible}
if for each pair of indices $i,j$ there is $k \in \bbN$ such that
$A^k_{i,j}>0$. If $k$ can be chosen independently of $i,j$
we call $A$ \emph{primitive}.
Every essential graph $\Gamma $ defines an \emph{edge shift} $X_\Gamma$.
This subshift of finite type over the alphabet $E(\Gamma)$ consists of all
biinfinite sequences of edges $(e_k)_{k \in \bbZ}$
with $\init(e_{k+1})=\term(e_{k})$ for all $k \in \bbZ$.

Up to renaming the edges,
the graph and hence the edge shift is uniquely determined
by the adjacency matrix.
Hence, for a graph $\Gamma$ with adjacency matrix $A$
we denote by $X_A$ the edge shift of $\Gamma$.
For $A \in \bbZP^{V \times V}$
we can canonically label the edges in the graph defined by $A$ from $i$ to $j$
by $(i,j,1),(i,j,2),\dots,(i,j,A_{i,j})  \in V \times V \times \bbN$.

An automorphism of a graph is a pair of bijections $\theta_V: V \to V$
and $\theta_E: E \to E$ such that $\theta_V \circ \init_\Gamma = \init_\Gamma
\circ \theta_E$
and $\theta_V \circ \term_\Gamma = \term_\Gamma \circ \theta_E$.
For essential graphs the map on the edges uniquely
determines the map on the vertices.
Every graph automorphism
induces an automorphism of the corresponding edge shift.
If a group $G$ is acting by graph automorphisms
on a graph $\Gamma$ we sometimes
do not write the action explicitly
and simply write $gi$ or $ge$ for
the image of the vertex $i$ or the edge $e$,
respectively, under the action of $g \in G$.

We will be interested in various objects associated to
a subshift $(X,\sigma)$. We denote
by
\begin{align*}
  \Per_k(X,\sigma):= \{x \in X \setsep \sigma^k(x)=x\}
\end{align*}
the set of \emph{$k$-periodic points of $X$}.
The cardinalities of these sets can be conveniently encoded
in the zeta function
\begin{align*}
  \zeta_X(t) := \exp(\sum_{k=1}^\infty \frac{\abs{\Per_k(X,\sigma)}}{k} t^k)
\end{align*}
Now let $A \in \bbZP^{V \times V}$ be essential.
The \emph{dimension group} $D_A$ of $X_A$
is defined as the direct limit
of the diagram \[\bbZ^V \xrightarrow{A} \bbZ^V \xrightarrow{A} \bbZ^V \xrightarrow{A}\dots\]
Formally, we define
$D_A$ as $\bbZ^V \times \bbZ / \sim$
where $\sim$ is the equivalence relation
defined by $(x,n) \sim (xA,n+1)$.
One can further endow this group
with an order and the automorphism induced
by the shift map to obtain the so called \emph{dimension triple}, but we don't need this
additional structure here.



If $A$ is an $M \times M$ matrix and $B$ is a $K \times K$ matrix,
the \emph{Kronecker product} $A \otimes B$ of $A$ and $B$ is the
$(M \times K) \times (M \times K)$ matrix with entries
$(A \otimes B)_{(i,k),(j,\ell)} = A_{i,j}B_{k,\ell}$.
If $M$ and $K$ are totally ordered and $M \times K$ is
endowed with the lexicographic ordering, then
this corresponds to the usual Kronecker product.
For a set $M$ we denote by $\ones_M$ the vector in $\bbZ^M$
whose entries are all equal to $1$
and we denote by $\ones_{M \times M} := \ones_M \ones_M^\top$
the $M \times M$ matrix which equals $1$ in every entry.
The $M \times M$ identity matrix is denoted by $I_{M \times M}$.

\section{Finite Groups Acting Freely on SFTs}
\label{sec:finite groups acting freely}

In the following let $G$ be a finite group whose identity
element we denote by $1_G$.
We start by introducing our central objects of interest.

\begin{defn}[$G$-SFT]
  \label{defn:g-sft}
  A \emph{free $G$-SFT} is a triple
  $(X,\sigma,\alpha)$
  where $(X,\sigma)$ is an SFT 
  and $\alpha: G \to \Aut(X,\sigma)$ is
  a free action of $G$ on $(X,\sigma)$.
  \enquote{Free} here means, that $\alpha(g)(x) \neq x$
  for all $x \in X$ and $g \neq 1_G$.
\end{defn}

\begin{defn}[$G$-conjugacy]
  \label{defn:g-conjugate}
  Two free $G$-SFTs $(Y_1,\sigma_1,\alpha_1)$
  and $(Y_2,\sigma_2,\alpha_2)$ are \emph{$G$-conjugate}, if
  there is a homeomorphism $\varphi:Y_1 \to Y_2$
  such that $\varphi \circ \sigma_1 = \sigma_2 \circ \varphi $
  and $\varphi \circ \alpha_1(g) =\alpha_2(g) \circ \varphi$
  for all $g \in G$.
\end{defn}

\begin{defn}[Eventual $G$-conjugacy]
  \label{defn:eventually-g-conjugate}
  Two free $G$-SFTs $(Y_{1},\sigma,\alpha_1)$
  and $(Y_{2},\sigma,\alpha_2)$ are \emph{eventually $G$-conjugate},
  if $(Y_{1},\sigma^\ell,\alpha_1)$
  and $(Y_{2},\sigma^\ell,\alpha_2)$ are
  $G$-conjugate for all sufficiently large $\ell$.
\end{defn}

The following proposition goes back to Franks, proofs can be found e.g. in
\cite[Proposition 2.9]{boyleAutomorphismGroupShift1988}
or \cite[Proof of Theorem 7.2]{saloTransitiveActionFinite2019}.

\begin{prop}[Representation of free $G$-SFTs by graph automorphisms]
  \label{thm:franks-repr}
    Let $G$ be a finite group.
  Every free $G$-SFT is $G$-conjugate to a $G$-SFT of the form
  $(X_A,\sigma,\alpha)$
  where $A$ is a square $\{0,1\}$ matrix
  defining an edge shift $(X_A,\sigma)$
  and $\alpha$ is induced by an action of $G$ by graph automorphisms
  of $\graph_A$
  acting freely on the vertex set.
\end{prop}
Every action of $G$ on an SFT $X$ induces a $G$-action
on the dimension group of this SFT. This is called
the \emph{dimension group representation} of the action of $G$ on $X$.
If $X=X_\Gamma$ is an edge shift and the action on $X_\Gamma$ is
induced by an action of $G$ on the graph
$\Gamma$, the
dimension group representation can be constructed as
follows.
Let $V$ be the vertex set and let $A$ be the adjacency matrix of
$\Gamma$.
Since $G$ acts by graph automorphisms,
we have $A_{gi,gj}=A_{i,j}$ for all $g \in G$ and $i,j \in V$.
Our group $G$ acts on $\bbZ^V$ by $w \mapsto gw$ with $(gw)_{i} = w_{g^{-1}i}$.
This action extends to an action
of $G$ on $D_A$ simply by
\begin{align*}
  [(w,n)] \mapsto [(gw,n)].
\end{align*}
This action is well-defined since $g(wA)_i = \sum_{j \in V} w_j A_{j,g^{-1}i} = \sum_{j \in V}
w_{g^{-1}j}A_{j,i} = ((gw)A)_i$.
Alternatively one can use the fact that the dimension group construction
is functorial in the sense that every isomorphism between SFTs induces a unique isomorphism of the
corresponding
dimension groups. This isomorphism can be either defined  directly
using Krieger's internal construction of the dimension group
using equivalence classes of rays as explained e.g. in \cite[Section
6]{boyleAutomorphismGroupShift1988}. Or else one uses the fact
that each isomorphism can be decomposed into a series of
state splittings and amalgamations. Between
edge shifts these can in turn be represented as
strong shift equivalence matrix pairs.
To show well-definedness of the resulting isomorphism,
i.e. independence from the decomposition into splittings and
amalgamations, one then uses Wagoner's complex.
This approach can be found in
\cite[Section 2]{wagoner1992type}
and \cite[Theorem 7.5.7]{lindIntroductionSymbolicDynamics2021}.
More concretely, if $A=RS$ and $B=SR$ is an elementary strong shift equivalence
between matrices $A$ and $B$, then
an element $[(w,n)]$ in the dimension group of $A$ is simply mapped to
$[(wR,n)]$, an element in the dimension group of $B$.
Again this is well-defined
since $wAR=wRB$.
In the case of a graph automorphism as above
we simply have $A=R_g R_{g^{-1}}A=R_{g^{-1}}A R_g$
where $R_g$ is the $V \times V$ permutation matrix
with $(R_g)_{i,j}=1$ iff $j=gi$,
so our action is given by $[(w,n)] \mapsto [(wR_g,n)]$
which agrees with the first definition
since $(wR_g)_{i}=w_{g^{-1}i}=(gw)_i$.

\begin{defn}[Inert free $G$-SFTs]
We say
that the free $G$-SFT $(X,\sigma,\alpha)$ is \emph{inert} if the action
on the dimension group induced by $\alpha$ is trivial,
i.e. $\alpha(g) \in \Aut(X,\sigma)$ is inert for all $g \in G$.
\end{defn}




Applying the explicit construction of the dimension group
representation for edge shifts we obtain the following
characterization of inertness,
which appeared as a remark 
in \cite[p. 497]{fiebigPeriodicPointsFinite1993b}. 

\begin{prop}[Inertness of free graph automorphisms]
  \label{prop:intertness-free-graph-autos}
  Let $A \in \bbZP^{V \times V}$ be essential and let
  $(X_A,\sigma,\alpha)$ be a free $G$-SFT induced by
  an action of $G$ on the graph defined by $A$. Then
  $(X_A,\sigma,\alpha)$ is inert, iff
  there is $\ell \in \bbN$ such that
  for all $g \in G,i,j \in V(G)$ we have
  \begin{align*}
    (A^\ell)_{i,j}=(A^{\ell})_{i,gj}
  \end{align*}
\end{prop}
\begin{proof}
  Let $A$ be an $V \times V$ matrix.
  Let $[(w,n)]$ be an element in the dimension group of $(X_A,\sigma)$.
  By the explicit construction of the dimension group representation,
  a group element $g$ fixes $[(w,n)]$ if
  only if
  there is $\ell \in \bbN$ such that
  $w A^\ell = (gw) A^\ell =g(w A^\ell)$.
  But this is equivalent to
  \begin{align*}
    (w A^\ell )_{gj} &= (g(wA^\ell))_{gj}=(wA^\ell)_{j}
  \end{align*}
  for all $j \in V$.
  This holds for all $w \in \bbZ^V$ if it holds for all standard basis vectors,
  which
  in turn is equivalent to
  \begin{align*}
    (A^\ell)_{i,j}=(A^\ell)_{i,g j}.
  \end{align*}
  for every $i,j \in V$.
\end{proof}

\section{Matrices Over The Integral Group Ring and $G$-Extensions of SFTs}
\label{sec:matrices-over-integer-group-rings}

In this section we look at $G$-SFTs from the opposite direction.
Instead of considering an SFT together with
a free $G$-action and obtaining
the orbit space as a quotient, we start with an
orbit space, which is itself an SFT,  and construct an SFT as a $G$-extension.
These extensions can be defined by
square matrix
over the integer group ring.
Our goal is then to show, that every free $G$-SFT
can be represented in that way.
We will use a very concrete definition of $G$-extensions.
For their connection to cocycles and the general
theory of extensions of dynamical systems, see \cite{boyleFiniteGroupExtensions2017}.

By $\bbZ[G]$ we denote the integer group ring
over $G$ consisting of all formal linear combinations
of elements in $G$ with integer coefficients.
The set $\bbZP[G]$ consists of all elements
in $\bbZ[G]$ whose coefficients are non-negative.

For $h \in G$ let $\pi_h$ be the map from $\bbZ[G] \to \bbZ$
defined by $\pi_h(\sum_{g \in G} a_g g)=a_h$.
When we apply this map entry wise, we also get a map
$\pi_h: \bbZ[G]^{V \times V} \to \bbZ^{V \times V}$
and we can write every matrix $B \in \bbZ[G]^{V \times V}$
as $B = \sum_{h \in G} \pi_h(B) h$.

Let $B \in \bbZP[G]^{V \times V}$.
We define two matrices $\calA(B) \in \bbZP^{V \times V}$
and $\calE(B) \in \bbZP^{(V \times G)\times (V \times G)}$
as follows:
\begin{align*}
  \calA(B) &:=\sum_{g \in G} \pi_g(B), \\
  \calE(B)_{(i,g),(j,h)} &:=\pi_{g^{-1}h}(B)_{i,j}.
\end{align*}
We call $\calA(B)$ the \emph{augmentation} of $B$
and $\calE(B)$ the \emph{extension} defined by $B$.

Using the Kronecker product, we can write this more concisely.
Let $P_g$ be the permutation matrix defined by
\begin{align*}
  (P_g)_{h,k} = \begin{cases}
    1 \text{ if } k=hg \\
    0 \text{ otherwise}
  \end{cases}.
\end{align*}
Then
\begin{align*}
  \calE(B) = \sum_{g \in G} \pi_g(B) \otimes P_g.
\end{align*}

We get an action $\theta_B$ of $G$ on the graph
defined by $\calE(B)$ as follows. The automorphism
$\theta_B(g)$ acts
on the vertex set $V \times G$ of this graph
simply by $\theta_B(g)(i,h)=(i,gh)$.
We can extend this map to the edges, and thus define a graph automorphism, since
\begin{align*}
  \calE(B)_{(i,rg),(j,rh)}
  =\pi_{g^{-1}h}(B)_{i,j}=\calE(B)_{(i,g),(j,h)}
  \text{ for every } r \in G.
\end{align*}
Therefore we can define that $\theta_B(r)$ maps the $k$-th edge between $(i,g)$ and $(j,h)$
to the $k$-th edge between $(i,rg)$ and $(j,rh)$.
This action by graph automorphisms immediately extends
to an action $\beta_B$ of $G$ on the edge shift $X_{\calE(B)}$.

The graphs defined by $B$, $\calA(B)$ and $\calE(B)$ are related as
follows: $B$ defines a graph $\Gamma$ with edge labels in $G$.
The graph defined by $\calA(B)$ simply forgets about the edge labels
and $\calE(B)$ defines a covering graph or more specifically a $|G|$-lift of $\Gamma$.

\begin{exam}[Augmentation and Extension]
  \label{exam:extension-and-augmentation}
  Consider $\bbZ/2\bbZ=\{e,g\}$, where $e$ is the identity element,
  and 
  \begin{align*}
    B=\begin{pmatrix}g & e \\ g &0\end{pmatrix}                      
  \end{align*}
  The labeled graph corresponding to $B$ is
  depicted on the left in \Cref{fig:graphBG}.
  We have
  \begin{align*}
    \calA(B)&=\begin{pmatrix}
                1 & 1 \\ 1 &0
              \end{pmatrix}\\
    \calE(B)&=\begin{pmatrix}
               0 & 1 \\ 0 &0
             \end{pmatrix}
              \otimes
             \begin{pmatrix}
               1 & 0 \\ 0 &1
             \end{pmatrix}
             +                             
              \begin{pmatrix}
               1 & 0 \\ 1 &0
             \end{pmatrix}
              \otimes
             \begin{pmatrix}
               0 & 1 \\ 1 &0
             \end{pmatrix} \\
            &=\blockmat{\begin{matrix}
               0 & 1 \\ 1 &0
             \end{matrix} &\begin{matrix}
               1 & 0 \\ 0 &1
             \end{matrix} \\
    \hline
             \begin{matrix}
               0 & 1 \\ 1 &0
             \end{matrix}
             &\begin{matrix}
               0 & 0 \\ 0 &0
             \end{matrix}}
  \end{align*}
  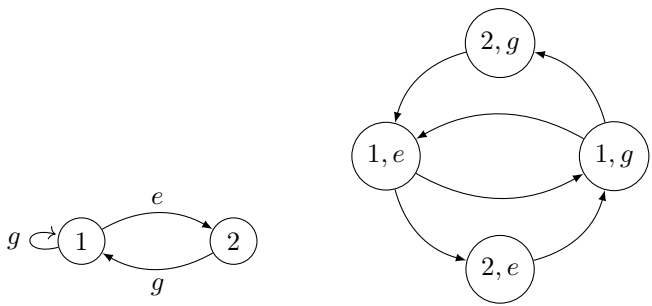
\begin{figure}[h]
    \begin{center}
    \begin{tikzpicture}
      \node[circle,draw] (1) at (-2,0) {1};
      \node[circle,draw] (2) at (0,0) {2};
      \draw[-latex] (1) edge[bend left, "$e$"] (2);
      \draw[-latex] (2) edge[bend left, "$g$"] (1);
      \draw[-latex] (1) edge[loop left, "$g$"] (1);
    \end{tikzpicture}
    \hspace{1cm}
    \begin{tikzpicture}
      \node[circle,draw] (01) at (-1.5,0) {$1,e$};
      \node[circle,draw] (10) at (1.5,0) {$1,g$};
      \node[circle,draw] (00) at (0,1.5) {$2,g$};
      \node[circle,draw] (11) at (0,-1.5) {$2,e$};
      \draw[-latex] (00) edge[bend right] (01);
      \draw[-latex] (01) edge[bend right] (10);
      \draw[-latex] (10) edge[bend right] (01);
      \draw[-latex] (01) edge[bend right] (11);
      \draw[-latex] (10) edge[bend right] (00);
      \draw[-latex] (11) edge[bend right] (10);
    \end{tikzpicture}
  \end{center}
    \caption{A $\bbZ/2\bbZ$ extension of the golden mean shift.}
    \label{fig:graphBG}
\end{figure}
\end{exam}

There is a canonical factor map from $X_{\calE(B)}$ to $X_{\calA(B)}$
induced by the graph homomorphism from $\Gamma_{\calE(B)}$ to
$\Gamma_{\calA(B)}$ that
maps the $\ell$-th edge from $(i,g)$ to $(j,h)$ 
onto the $\ell$-th edge from $i$ to $j$ labeled by $g^{-1} h$.
We now want to describe the preimages
of $(e_k)_{k \in \bbZ} \in X_{\calA(B)}$ under
this factor map.
To do so, we have to introduce slightly more notation.
We may relabel the edges from $i$ to $j$ in the graph
defined by $\calA(B)$ by
$(i,j,t,m)$ with $t \in G$ and $m \in \{1,\dots,\pi_t(B_{i,j})\}$.
The preimages of $(i_k,j_k,t_k,m_k)_{k \in \bbZ}$
under the factor map described above are of the form
\[((i_k,g_k),(j_k,h_k),m_k)_{k \in \bbZ}\] 
where $g_0$ is any element in $G$ and where
$g_{k+1}=h_k=g_k t_k$.
These preimages are precisely the $G$-orbits in $X_{\calE(B)}$ of
the action $\beta_B$.
Therefore $(X_{\calA(B)},\sigma)$ is conjugate to
the orbit space $(X_{\calE(B)}/\beta_B,\sigma)$.

Parry showed that every free $G$-SFT is isomorphic
to a $G$-extension defined by a matrix over $\bbZP[G]$.
This construction is sketched in \cite[Section
2]{boyleEquivariantFlowEquivalence2005}.
For completeness we give a short explicit proof using the notation introduced above
in the next proposition.

\begin{prop}[{Representing free $G$-SFTs by matrices over $\bbZP[G]$}]
  \label{prop:repr-by-group-ring-matrix}
  For every free $G$-SFT $(X,\sigma,\alpha)$ we
  can find $B \in \bbZP[G]$ such that
  \[(X,\sigma,\alpha) \cong (X_{\calE(B)},\sigma,\beta_B).\]
\end{prop}
\begin{proof}
  By \Cref{thm:franks-repr} we may assume that $X=X_A$ is an edge
  shift
  defined by a $\{0,1\}$-matrix $A$ and $\alpha$ is induced
  by an action $\eta$ of $G$ by graph automorphisms acting freely on the
  vertex set $V$ of the graph $\Gamma$ defined by $A$.

  Under these assumptions we will relabel the vertices and edges of $\Gamma$ such that
  $(X_{\calE(B)},\sigma,\beta_B)$ is  equal
  to $(X_A,\sigma,\alpha)$ as follows.
  Select vertices $v_{1,e},\dots,v_{m,e}$, one from each $G$-orbit.
  Since $\eta$ acts freely on the vertices and
  every vertex is contained in precisely one $G$-orbit, for each $v \in V$ there is now a unique pair
  $(i,g)$ s.t. $v=\eta(g)(v_{i,e})$.
  We relabel the vertex $v$ by $v_{i,g}$.
  For $g \in G$ define $B_g \in \{0,1\}^{m \times m}$ by $(B_g)_{i,j}=1$ iff
  there is an edge from $v_{i,h}$ to $v_{i,hg}$.
  Set $B=\sum_{g \in G} B_g g$.

  Now if we label the vertices in $V$ with $(i,g)$ instead of
  $v_{i,g}$ and thus turn $A$ into a $(m \times G) \times (m \times
  G)$ matrix, we really get an equality
  $\calE(B) = A$ and $\alpha = \beta_B$.
\end{proof}

The operations of forming $\calA(B)$ and $\calE(B)$ behave nicely with
respect to matrix powers as the next lemma shows.
\begin{lem}[Compatibility of matrix multiplication with augmentation
  and extension]
  \label{lem:compatibility}
  For $B \in \bbZP[G]^{V \times V}$ and $k \in \bbN$ we have
  \begin{align*}
    \calA(B^k)&=\calA(B)^k\\
    \calE(B^k)&=\calE(B)^k.
  \end{align*}
\end{lem}
\begin{proof}
  This can be shown by induction since
  \begin{align*}
    \calA(B^{k+1})_{i,j} &= \sum_{g \in G} \sum_{\ell \in V} \pi_g( 
    (B^k)_{i,\ell}B_{\ell,j})\\
    &=  \sum_{\ell\in V} \sum_{g \in G} \sum_{h \in G} \pi_h( 
      (B^k)_{i,\ell}) \pi_{h^{-1}g}(B_{\ell,j}) \\
    &=  \sum_{\ell \in V} \sum_{h \in G} \pi_h( 
      (B^k)_{i,\ell}) \sum_{g \in G}  \pi_{g}(B_{\ell,j})\\
    &=  \sum_{\ell \in V} \calA(B)^k_{i,\ell} \calA(B)_{\ell,j}=
      \calA(B)^{k+1}_{i,j}.
  \end{align*}
  Similarly
  \begin{align*}
    \calE(B^{k+1})_{(i,s),(j,t)} &= \pi_{s^{-1}t} (B^{k+1}_{i,j}) \\
                                 &=\sum_{\ell \in V} \pi_{s^{-1}t} (B^{k}_{i,\ell} B_{\ell,j})\\
                                 &=\sum_{\ell \in V} \sum_{g \in G} 
                                   \pi_{s^{-1}g}(B_{i,\ell}^k)
                                   \pi_{g^{-1}t}(B_{\ell,j})\\
    &=\sum_{\ell \in V} \sum_{g \in G}\calE(B)^k_{(i,s),(\ell,g)}
      \calE(B)_{(\ell,g),(j,t)} = \calE(B)^{k+1}_{(i,s),(j,t)}.\qedhere
  \end{align*}
\end{proof}

The usual definition of shift equivalence
due to Williams
for integer matrices carries over to the case of matrices
over rings or subsets of rings, see
\cite{boyleFiniteGroupExtensions2017}.

\begin{defn}[Shift equivalence and strong shift equivalence]
    \label{defn:gr-shift-equivalence}
    Let $\calR$ be a subset of a ring.
  We say that two matrices $A \in \calR^{V \times V}$
  and $B \in \calR^{W \times W}$ are \emph{shift equivalent (SE)
  over $\calR$ with lag $\ell \in \bbN$} if there are
  matrices $R \in \calR^{V \times W}$ and $S \in \calR^{W \times V}$ 
  and $\ell \in \bbN$ such that
  $A^\ell = RS$, $B^\ell = SR$
  and $AR=RB, SA=BS$.
  We say that $A$ and $B$ are \emph{elementary strong
  shift equivalent} over $\calR$ if
  they are shift equivalent with lag $1$.
  We say that $A$ and $B$ are \emph{strong shift equivalent (SSE) over $\calR$}
  if there is a chain of matrices $A=C_1,\dots,C_n=B$
  such that $C_k$ is elementary strong shift equivalent
  to $C_{k+1}$ for each $k \in \{1,\dots,n-1\}$.
\end{defn}

By a famous theorem due to Williams, Kim and Roush, strong shift equivalence
of non-negative integer matrices is equivalent to
conjugacy of the corresponding edge shifts
and shift equivalence of the matrices is equivalent to
eventual conjugacy.

The same theorem holds for matrices over $\bbZP[G]$
and free $G$-SFTs. This was shown by Parry, Sullivan, Boyle and Schmieding,
see \cite[Proposition B.11]{boyleFiniteGroupExtensions2017}
and \cite[Proposition 2.7.1.]{boyleEquivariantFlowEquivalence2005}.
\begin{thm}[Strong shift equivalence and conjugacy]
  \label{thm:gr-se-eventual-conjugacy}
  Let $B \in \bbZP[G]^{V \times V}$
  and $C \in \bbZP[G]^{W \times W}$. Then the following are equivalent
  \begin{enumerate}
  \item $B$ and $C$ are SSE over $\bbZP[G]$.
  \item $(X_{\calE(B)},\sigma,\beta_B)$ and
    $(X_{\calE(C)},\sigma,\beta_C)$
    are $G$-conjugate.
  \end{enumerate}
\end{thm}

\begin{thm}[Shift equivalence and eventual conjugacy]
  \label{thm:gr-se-eventual-conjugacy}
  Let $B \in \bbZP[G]^{V \times V}$
  and $C \in \bbZP[G]^{W \times W}$. Then the following are equivalent
  \begin{enumerate}
  \item $B$ and $C$ are SE over $\bbZP[G]$.
  \item $(X_{\calE(B)},\sigma,\beta_B)$ and
    $(X_{\calE(C)},\sigma,\beta_C)$
    are eventually $G$-conjugate.
  \end{enumerate}
\end{thm}

\section{A Characterization of Free Inert $G$-SFTs}
\label{sec:inert-g-sfts}
In this section we collect the results about
inert free $G$-SFTs which we need in the proof of our main result.
An important role is played by the element $u_G : = \sum_{g \in G} g \in \bbZ_+[G]$.

\begin{defn}[{Inert matrices over $\bbZP[G]$}]
We call a matrix $B \in \bbZP[G]^{V \times V}$ \emph{inert},
if the $G$-SFT $(X_{\calE(B)},\sigma,\beta_B)$ is inert.
\end{defn}

\begin{thm}[Characterization of inert $G$-SFTs]
  \label{thm:charact-inert}
  Let $B \in \bbZP[G]^{V \times V}$.
  Then the following are equivalent:
  \begin{enumerate}[(i)]
  \item $B$ is inert. \label{enum:inert-1}
  \item For every $g,s,t \in G$, $i,j \in V$
    and every sufficiently large $\ell \in \bbN$
    we have\label{enum:inert-2}
    \begin{align}
      \calE(B^\ell)_{(i,s),(j,t)} &=
      \calE(B^\ell)_{(i,s),(j,gt)}.
      \label{eq:inert-charac-2}
    \end{align}
    \item For all sufficiently large $\ell \in \bbN$ we have
      $B^{\ell} \in u_G \bbZP^{V \times V}$.\label{enum:inert-3}
    \item $\calE(B)$ and $\calA(B)$ are shift equivalent.
    \item $\zeta_{\calE(B)}= \zeta_{\calA(B)}$.
  \end{enumerate}
\end{thm}
\begin{proof}~\\ \\
  \noindent $(i) \Leftrightarrow (ii)$: This is
  a direct consequence of \Cref{lem:compatibility} and
  \Cref{prop:intertness-free-graph-autos}.\\
  
  \noindent(ii) $\Leftrightarrow$ (iii): This equivalence follows directly from 
  $\calE(B^\ell)_{(i,s),(j,t)} = \pi_{s^{-1}t}(B^\ell)_{i,j}$
  as follows:
  Using this equation we can reformulate \eqref{eq:inert-charac-2}
  as:  For every $g,h_1,h_2 \in G$ we have
  $\pi_{h_1^{-1}h_2}(B^\ell)=\pi_{h_1^{-1}gh_2}(B^\ell)$,
  which in turn is equivalent to
  $\pi_g(B^\ell)=\pi_{1_G}(B^\ell)$ for all $g \in G$.
  But this is nothing else than $B^\ell \in u_G \bbZP^{V \times V}$.\\
  
  \noindent(iii) $\Rightarrow$ (iv):
    Let $\ell \in \bbN$ such that
  $B^\ell \in u_G \bbZP^{V \times V}$. Then
  \begin{align*}
    \calE(B^\ell) &= \sum_{g \in G} \pi_g(B^\ell) \otimes P_g = \pi_{1_G}(B^\ell) \otimes \sum_{g \in G} P_g\\
             &=\pi_{1_G}(B^\ell) \otimes
              \ones_{G \times G}\\
             &=(\pi_{1_G}(B^\ell)  \otimes \ones_G)(I_{V \times V} \otimes \ones_G^\top)\\
    \calA(B^\ell) &= \sum_{g \in G} \pi_g(B^\ell)  = |G| \pi_{1_G}(B^\ell) =
               \pi_{1_G}(B^\ell) \otimes \mat{|G|} \\
             &= (I_{V \times V} \otimes \ones_G^\top)
                                           (\pi_{1_G}(B^\ell)  \otimes \ones_G).
  \end{align*}
  Furthermore
  \begin{align*}
    \sum_{g \in G}\pi_{g}(B) \pi_{1_G}(B^{\ell})
    &= \sum_{g \in G} \pi_{g}(B)\pi_{g^{-1}}(B^{\ell}) \\
    &= \pi_{1_G}(B^{\ell+1})\\
    &= \sum_{g \in G} \pi_{g^{-1}}(B^{\ell})\pi_{g}(B) \\
    &= \pi_{1_G}(B^{\ell})\sum_{g \in G}\pi_{g}(B).
  \end{align*}
  Setting $R:=I_{V \times V} \otimes \ones_G^\top$
  and $S:=\pi_{1_G}(B^\ell)  \otimes \ones_G$
  we therefore have
  \begin{align*}
    \calA(B)^\ell &= RS, \\
    \calE(B)^\ell &= SR,\\
    \calA(B) R &= (\sum_{g \in G} \pi_g(B) \otimes 1) (I_{V \times V} \otimes
                 \ones_G^\top) \\
                  &= \sum_{g \in G} \pi_g(B) \otimes  \ones_G^\top\\
                  &= (I_{V \times V} \otimes \ones_G^\top) (\sum_{g \in G} \pi_g(B)
      \otimes P_g)
      =R \calE(B), \\
  S \calA(B) &= (\pi_{1_G}(B^\ell) \otimes
                \ones_G ) (\sum_{g \in G} \pi_g(B) \otimes 1) \\
    &= (\sum_{g \in G} \pi_g(B)
      \otimes P_g)  (\pi_{1_G}(B^\ell) \otimes \ones_G)
      =\calE(B) S.
  \end{align*}
  Therefore the pair $(R,S)$ defines a shift equivalence with lag $\ell$ between
  $\calA(B)$
  and $\calE(B)$.\\

  \noindent(iv) $\Rightarrow$ (v): The zeta function is an invariant of
  shift equivalence,
  see e.g. \cite[Corollary 7.4.12]{lindIntroductionSymbolicDynamics2021}.

  \noindent(v) $\Rightarrow$ (ii): This result as well as the converse
  direction are due to Fiebig \cite[Theorem B]{fiebigPeriodicPointsFinite1993b}.
  We give a simplified proof, since this allows us to also get
  a quantitative bound on the lag of the shift equivalence in (iii).
  Let $(w_g)_{g \in G}$ be an orthogonal basis
  of $\bbC^G$ such that $w_{1_G}=\ones_G$.
  Let $F$ be the $G \times G$ matrix
  with columns $(w_g)_{g \in G}$. 
  For every $g \in G$ we have $P_g w_{1_G} = w_{1_G}$ and $w_{1_G}^\top P_g = w_{1_G}^\top$
  and hence
  \begin{align*}
    F^{-1} P_g F = \blockmat{1 & 0 \\ \hline 0 & Q_g}
  \end{align*}
  for some matrix $Q_g \in \bbC^{\tilde{G} \times \tilde{G}}$
  with $\tilde{G} = G \setminus \{1_G\}$.
  Therefore
  \begin{align*}
    (I_{V \times V} \otimes F)^{-1} \calE(B) (I_{V \times V} \otimes F)
    &= \sum_{g \in G} \pi_g(B) \otimes (F^{-1} P_g F) \\
    &= \sum_{g \in G} \pi_g(B) \otimes  \blockmat{1 & 0 \\ \hline 0 & Q_g}
  \end{align*}
  Calculating the characteristic polynomial of
  the resulting block diagonal matrix gives
  $\chi_{\calE(B)}(t) = \chi_{\calA(B)}(t) \chi_{R}(t)$
  where $Q:= \sum_{g \in G} \pi_g(B) \otimes Q_g$.
  Since the characteristic polynomials of $\calE(B)$
  and $\calA(B)$ agree with the inverse of the
  zeta function up to a power of $t$,
  we get $\chi_{Q}(t)=t^k$ where
  $k$ is the size of $Q$. Hence $Q$
  is nilpotent and we get
  \begin{align*}
    (I_{V \times V} \otimes F)^{-1} \calE(B^k) (I \otimes F)
    &= \calA(B^k) \otimes  \blockmat{1 & 0 \\ \hline 0 & 0}
  \end{align*}
  and hence
  \begin{align*}
     \calE(B^k)
    &= \calA(B^k) \otimes  ( F\blockmat{1 & 0 \\ \hline
    0 & 0} F^{-1}) \\
    &=\calA(B^k) \otimes
      \ones_{G \times G}.
  \end{align*}
  In other words, condition (ii) is satisfied.
\end{proof}
\begin{rem}
  \begin{enumerate}[(a)]
    \item
  Notice that in \eqref{enum:inert-3} we can replace
  \enquote{For all sufficiently large $\ell \in \bbZP$}
  by \enquote{For some $\ell \in \bbZP$} since $u_G \bbZP^{V\times V}$
  is closed under multiplication with matrices in $\bbZP[G]^{V \times
    V}$.
\item The last part of the proof in Theorem 5.1 furthermore tells
  us how large we have to choose the power $\ell$
  in (ii) and (iii). Namely, we need $\ell \geq n(|G|-1)$,
  since this is the size of $Q$ and hence an upper bound on its
  nilpotency
  index.
\end{enumerate}
\end{rem}

As a last ingredient we also need the fact, that for square matrices  $B,C$ over the integer group ring
whose entries eventually
lie in $u_G\bbZP$,
we can lift a SE between $\calA(B)$ and $\calA(C)$
to a SE between $\calE(B)$ and $\calE(C)$.
This was shown by Boyle and Schmieding in \cite{boyleFiniteGroupExtensions2017}.
The key observation is that for a matrix $B$ with
entries in $\bbZP[G]$ we have
$u_G B = u_G \calA(B)$.
\begin{lem}[{\cite[Lemma 4.4]{boyleFiniteGroupExtensions2017}}]
  \label{thm:lift-SE}
  Consider matrices $B \in \bbZP[G]^{V\times V}$ and $C \in \bbZP[G]^{W \times W}$
  such that for all sufficiently large $k$
  we have $B^k \in u_G\bbZP^{V \times V}$, $C^k \in u_G \bbZP^{W \times W}$.
  If $\calA(B)$ and $\calA(C)$ are SE over $\bbZP$,
  then $B$ and $C$ are SE over $\bbZP[G]$.
\end{lem}
\begin{proof}
  Assume there is a shift equivalence over $\bbZP$ with
  lag $k$ given by the matrix pair $R,S$
  and assume that
  both $B^k$ and $C^k$ have entries in $u_G \bbZP$.
  Then both $\calA(B^k)$ and $\calA(C^k)$
  have all entries divisible by $\calA(u_g)=|G|$.
  Furthermore $B^k=u_G \frac{1}{|G|}\calA(B^k) $
  and $C^k=u_G\frac{1}{|G|} \calA(C^k)$.
  Hence $(u_G R)$ and $(\frac{1}{\abs{G}}\calA(C^k) S)$
  establish a shift equivalence of lag $2k$ between $B$ and $C$.
\end{proof}
   
\section{Proof of the Main Theorem}

With the preparation from the previous sections we are now ready to
prove our main theorems.

\begin{thm}[Main theorem - algebraic version]
  \label{thm:main-algebraic-version}
  Every pair of inert square matrices over $\bbZP[G]$,
  whose augmentations are shift equivalent over $\bbZP$, is shift equivalent over $\bbZP[G]$ itself.
\end{thm}
\begin{proof}
  Let $B,C$ be two inert matrices over $\bbZP[G]$
  such that $\calA(B)$ and $\calA(C)$ are SE over $\bbZP$.
  Since $B$ and $C$ are inert, we know by
  \Cref{thm:charact-inert}
  that for sufficiently large $\ell$, $B^\ell$ and $C^\ell$
  have entries in $u_G \bbZP$. Therefore we can apply
  \Cref{thm:lift-SE} and conclude that $B$ and $C$ are SE
  over $\bbZP[G]$.
\end{proof}

\begin{thm}[Main theorem - dynamical version]
  \label{thm:main}
  Let $(Y_1,\sigma)$ and $(Y_2,\sigma)$ be
  SFTs such that $(Y_1,\sigma^\ell)$, $(Y_2,\sigma^\ell)$ are
  conjugate
  for all sufficiently large $\ell$.
  Let $\alpha_1$ and $\alpha_2$ be free inert actions of a finite
  group $G$ on $(Y_1,\sigma)$ and $(Y_2,\sigma)$, respectively.
  Then for every sufficiently large $\ell$
  there is a conjugacy $\varphi$ from $(Y_1,\sigma^\ell)$
  to $(Y_2,\sigma^\ell)$ such that $\varphi \circ \alpha_1(g) =
  \alpha_2(g) \circ \varphi$ for all $g \in G$. 
\end{thm}
\begin{proof}
  By assumption $(Y_1,\sigma,\alpha_1)$ and $(Y_2,\sigma,\alpha_2)$ are
  free inert $G$-SFTs for which $(Y_1,\sigma)$ and $(Y_2,\sigma)$ are
  eventually conjugate.  We want to show that $(Y_1,\sigma,\alpha_1)$
  and $(Y_2,\sigma,\alpha_2)$ are eventually $G$-conjugate.

  By \Cref{prop:repr-by-group-ring-matrix} we find
  matrices $B_i \in \bbZP[G]^{{V_i \times V_i}}, i \in \{1,2\}$
  such that $(Y_i,\sigma,\alpha_i)$  is $G$-conjugate to
  $(X_{\calE(B_i)},\sigma,\beta_{B_i})$.
  In particular $(Y_i,\sigma)$ and $(X_{\calE(B_i)},\sigma)$
  are conjugate and therefore
  $(X_{\calE(B_1)},\sigma)$ and $(X_{\calE(B_2)},\sigma)$ are eventually conjugate.
  Algebraically this means that $\calE(B_1)$ and $\calE(B_2)$ are SE.
  Since $(Y_i,\sigma,\alpha_i)$ for $i=1,2$ is inert, so
  is $(X_{\calE(B_i)},\sigma,\beta_{B_i})$, in other words,
  $B_i$ is inert.
  By \Cref{thm:charact-inert} this implies
  that all four matrices $\calE(B_1),\calE(B_2),\calA(B_1)$ and
  $\calA(B_2)$
  are SE over $\bbZP$. By \Cref{thm:lift-SE}
  this implies that $B_1$ and $B_2$ are SE over $\bbZP[G]$.
  Finally by \Cref{thm:gr-se-eventual-conjugacy}
  this is equivalent to $(X_{\calE(B_1)},\sigma,\beta_{B_1})$
  and $(X_{\calE(B_2)},\sigma,\beta_{B_2})$
  being eventually $G$-conjugate.
  Since $(X_{\calE(B_i)},\sigma,\beta_{B_i})$
  is $G$-conjugate to $(Y_i,\sigma,\alpha_i)$
  this finally proves the theorem.
\end{proof}

\begin{cor}[Eventual conjugacy for free finite order automorphisms of full shifts]
  \label{thm:main-corr-classic-version}
  Let $k,m \in \bbN$.
  Let $(X_k,\sigma)$ be the full $k$-shift and 
  let $\varphi_1,\varphi_2$ be two automorphisms
  of $(X_k,\sigma)$ for which every orbit has size $m$.
  Then there is a homeomorphism $\psi: X_k \to X_k$
  such that  $\psi \circ \varphi_1 = \varphi_2 \circ \psi$
  and $\psi \circ \sigma^\ell = \sigma^\ell
  \circ \psi$ for all sufficiently large $\ell \in \bbN$.
\end{cor}
\begin{proof}
  The fact that all orbits of $\varphi_i, i \in \{1,2\}$
  have the same size $m$ is equivalent to the fact
  that $\alpha_i(k)(x) = \varphi_i^k(x)$ defines
  a free $\bbZ/m\bbZ$ action on $(X_k,\sigma)$.
  It is well-known that
  every finite-order automorphism of a full shift is inert.
  This follows from the fact that the automorphism
  group of the dimension triple
  is of the form $\bbZ^n$ and thus torsion-free,
  see e.g. \cite[Proposition 2.4]{hartmanStabilizedAutomorphismGroup2021}

  Thus we can apply the previous theorem
  to $(X_k,\sigma,\alpha_1)$ and $(X_k,\sigma,\alpha_2)$
  and obtain our result.
\end{proof}

\begin{cor}[Conjugacy of free finite order elements in the stabilized automorphism group]
  \label{cor:main-corr-stabilized-version}
   Let $k,m \in \bbN$.
   Let $(X_k,\sigma)$ be the full $k$-shift.
   Let $\varphi_1, \varphi_2$ be two elements of the
   stabilized automorphism group of $(X_k,\sigma)$
   for which every orbit has size $m$.
   Then $\varphi_1$ and $\varphi_2$ are conjugate
   in $\Aut^\infty(X_k,\sigma)$.
\end{cor}
\begin{proof}
  This follows directly from the fact that there is $k \in \bbN$
  such that $\varphi_1,\varphi_2 \in \Aut(X_k,\sigma^k)$
  and the previous theorem.
\end{proof}

\section{A Theorem of Kim and Roush}
\label{sec:kim-roush}

\Cref{thm:charact-inert} gives a new way to interpret a theorem by Kim and Roush
from \cite{kimFreeSbActions1997}.
Namely, in our language their theorem characterizes the existence
of an inert $\bbZ/p\bbZ$ extension of a given mixing SFT.

Recall that for a subshift $X$ we denote by
$\Per_k(X,\sigma) = \{x \in X \setsep \sigma^k(x)=x\}$
the $\sigma$-periodic points in $X$ with period $k$.
Denote by $\widetilde{\Per}_k(X,\sigma)$ the set
of $\sigma$-periodic points in $X$ with minimal period $k$,
i.e. $\widetilde{\Per}_k(X,\sigma) := \Per_k(X,\sigma) \setminus
\bigcup_{\ell<k} \Per_\ell(X,\sigma)$.
Endow $\Per(X) := \bigcup_{k \in \bbN} \Per_k(X,\sigma)$
with the discrete topology.

\begin{thm}[{\cite[Theorem 7.2 and Lemma 2.2]{kimFreeSbActions1997}}]
  Let $A$ be a primitive square matrix over $\bbZP$ and 
  set $o_\ell := \abs{\widetilde{\Per}_\ell(X_A,\sigma)}$.
  Let $p$ be prime.
  Then the following are equivalent:
  \begin{enumerate}[(i)]
  \item $X_A$ has an inert $\bbZ/p \bbZ$ extension, i.e.
    there is an inert square matrix $B$
    over $\bbZP[\bbZ/p \bbZ]$ whose augmentation $\calA(B)$
    is equal to $A$.
  \item  $X_A$ satisfies the so-called \enquote{Boyle-Handelmann} condition,
    i.e.
    there is a free $\bbZ/p \bbZ$ action
    on $(\Per(X_A),\sigma)$ such that the quotient by
    this action is conjugate to $(\Per(X_A),\sigma)$.
  \item For every $n \in \bbN$ the numbers $(o_\ell)_{\ell \in \bbN}$
    satisfy the following condition:
    \begin{align*}
      \sum_{k=1}^{m} \frac{p-1}{p^k} o_{n/p^k} \in \{0, \dots, o_{n}\}
    \end{align*}
    where $m = \max\{ k \in \bbN \setsep p^k \mid n\}$.
  \end{enumerate}
\end{thm}

The Boyle-Handelmann condition first appeared in a special case
in \cite[6.2 \enquote{Degrees
  example}]{boyleSpectraNonnegativeMatrices1991}.
Its necessity follows directly from the invariance of
the zeta function under shift equivalence.
In \cite[Lemma 2.2]{kimFreeSbActions1997} Kim and Roush showed that
it can be formulated arithmetically as in (iii).
The main work in the proof of Kim and Roush goes
into (iii) $\Rightarrow$ (i), which is an impressive technical tour de
force
using the polynomial representation
of SFTs.

Surprisingly, this theorem is basically everything that is known about the
existence of inert extensions and actions. A few other results
on inert actions can be found in \cite{longFixedPointShifts2009} and \cite{boyleMappingClassGroup2018}.

\section{Shift Equivalence Over $\bbZP[G]$ and Equivariant Flow
 Equivalence}
\label{sec:equivariant-flow-equivalence}

In \cite{boyleFlowEquivalenceGSFTs2020a}
Boyle, Carlsen and Eilers 
showed the following theorem as an application of their
characterization of $G$-equivariant flow equivalence.
For the somewhat lengthy definition
of $G$-equivariant flow equivalence
see e.g. \cite[Section 2]{boyleEquivariantFlowEquivalence2005}.
We only need here that it is an equivalence
relation between subshifts, which is weaker than
topological conjugacy and can algebraically be characterized
as in \Cref{thm:equivariant-flow-equivalence-characterization}.

\begin{thm}[\cite{boyleFlowEquivalenceGSFTs2020a}]
  \label{thm:flow-equ}
  Let $X$ be the full 
  $2k$-shift for some $k \in \bbN$.
  Let $(X,\sigma,\alpha)$ and $(X,\sigma,\beta)$
  be two free $\bbZ / 2\bbZ$-SFTs. Then there
  is a $\bbZ / 2\bbZ$-equivariant flow equivalence 
  between $(X,\sigma,\alpha)$ and $(X,\sigma,\beta)$.
\end{thm}

We will show in \Cref{thm:SE-implies-FE}
that at least for cyclic groups and irreducible free $G$-SFTs,
equivariant shift equivalence
already implies equivariant flow equivalence.
As a consequence we obtain \Cref{cor:flow-equ} as a generalization of
\Cref{thm:flow-equ}. This echos the recent result by Boyle
\cite{boyleShiftEquivalenceImplies2024}
that shift equivalence implies flow equivalence between SFTs in the ordinary non-equivariant
setting even without the irreducibility assumption.

A complete set of invariants for $G$-invariant
flow equivalence in the irreducible case was given by Boyle and Sullivan
already in \cite{boyleEquivariantFlowEquivalence2005}.

Their main result is the following theorem. All relevant notions
will be briefly defined afterwards. For a slightly simplified version treating the case of irreducible extensions see
\cite[Theorem 8.5]{boyleMappingClassGroup2018}.
\begin{thm}[{\cite[Theorem 6.4]{boyleEquivariantFlowEquivalence2005}}]
  \label{thm:equivariant-flow-equivalence-characterization}
  Let $G$ be a finite group. Let $A$ and $B$ be irreducible square
  matrices over $\bbZP[G]$. Assume that
  $A$ and $B$ have the same weight class. Let $H$
  be a group in the weight class. Let
  $A'$ and $B'$ be square matrices over $\bbZ_+[H]$
  which are SSE over $\bbZ_+[G]$ to $A$ and $B$
  respectively. Then the following are equivalent:
  \begin{enumerate}[(i)]
  \item $(X_{\calE(A)},\sigma,\alpha_A)$ and $(X_{\calE(B)},\sigma,\alpha_B)$
    are $G$-flow equivalent.
  \item There exists $g \in G$ such that $g H g^{-1} =
    H$
    and there is an $\EL_\infty(\bbZ[H])$ equivalence from
    $(I-A')_\infty$ to $(I-g B' g^{-1})_\infty$, i.e.
    there are $U,V \in \EL_\infty(\bbZ[H])$ such that
    \[U(I-A')_\infty V=(I-g B' g^{-1})_\infty.\]
    \label{cond:equ-flow-eq-chara-condition}
  \end{enumerate}
\end{thm}
Here are the necessary definition to make sense of this theorem:
For a matrix $A \in \bbZ[G]^{V \times V}$ and
$i \in V$ define $W_i(A)$
as the subgroup of $G$ consisting
of elements
of the form $\{g \in G \setsep \exists n \in \bbN:
\pi_g(A^n_{i,i})>0\}$.
These are all elements of $G$ which appear
as products over the labels along a cycle based at $i$ in the graph
defined by $A$.
We call such a product the \emph{weight} of the cycle.
The \emph{weight class} of $A$ is then defined as the conjugacy class
of $W_i(A)$ and this conjugacy class is independent of the choice of
$i$.
Recall (see e.g. \cite[Definition 3.1]{boyleEquivariantFlowEquivalence2005})
that a square matrix $A \in \bbZP[G]^{V \times V}$ is \emph{irreducible},
if for every $i,j \in  V$ there is $k \in \bbN$ such that $A^k_{i,j}
\neq 0$.
For a matrix $A \in \bbZ[G]^{n \times n}$
we define $A_\infty$ as the $\bbN \times \bbN$ matrix
with
\begin{align*}
  (A_\infty)_{i,j} =\begin{cases}
    A_{i,j} &\text{ if } i,j \leq n\\
    0 &\text{ if } \max(i,j)>n,\; i\neq j \\
    1 &\text{ if } \max(i,j)>n,\; i=j
  \end{cases}
\end{align*}
In other words, we take the  $\bbN \times \bbN$ identity matrix
and replace the upper left $n \times n$ block by $A$.
For $z \in \bbZ[G], i,j \in \bbN, i \neq j$ we denote by $E_{i,j}(z)$
the $\bbN \times \bbN$ matrix which
agrees with the identity matrix everywhere besides the entry at
position $i,j$,
where it equals $z$.
We say that two square matrices $A$ and $B$ over
$\bbZP[G]$ are \emph{elementarily positively equivalent}
if there are $g \in G$
and $i,j \in \bbN$ with
$\pi_g(A_{i,j})>0$ and $(I-B)_\infty=E_{i,j}(g)(I-A)_\infty$
or $(I-B)_\infty=(I-A)_\infty E_{i,j}(g)$.
We call the equivalence relation on square matrices
over $\bbZP[G]$ generated
by this notion \emph{positive equivalence}.
The stabilized general linear group
over $\bbZ[G]$ is defined as
$\GL_\infty(\bbZ[G]):= \{G_\infty \setsep G \in \GL_n(\bbZ[G])\}$.
The subgroup generated by the elementary matrices $\{E_{i,j}(z)
\setsep z \in \bbZ[G],\; i,j \in \bbN\}$
is denoted by $\EL_\infty(\bbZ[G])$.

In addition to the results appearing in
\cite{boyleFlowEquivalenceGSFTs2020a}
we also need the following two lemmas concerning the weight class.
\begin{lem}[Shift equivalent matrices have the same weight class]
  \label{lem:se-have-same-weight-class}
  Let $A$ and $B$ be irreducible square matrices over $\bbZP[G]$.
  If $A$ and $B$ are shift equivalent over $\bbZP[G]$, then $W(A)=W(B)$.
\end{lem}
\begin{proof}
  If $A$ and $B$ are shift equivalent, then there is $k \in \bbN$ such
  that
  $A':=A^{k|G|+1}$ and $B':=B^{k|G|+1}$ are irreducible and strong shift equivalent
  and therefore $W(A')=W(B')$ by
  \cite[Proposition 2.7.1, Theorem 3.3 and Proposition
  4.2]{boyleEquivariantFlowEquivalence2005}.
  It remains to show that $W(A)=W(A')$ and $W(B)=W(B')$.
  It follows directly from the definition
  that every weight of a cycle in $A'$ based at $i$
  also appears as a weight of a cycle based at $i$ in $A$.
  Therefore $W_i(A') \subseteq W_i(A)$.
  Now let $g$ be an element of $G$ with
  $\pi_g(A^\ell_{i,i}) > 0$ for some $\ell \in \bbN$.
  Then $g = g^{k|G|+1}$ and so $\pi_g(A^{\ell(k|G|+1)})_{ii}\neq 0$.
  Therefore $g \in W_i(A')$ and thus $W(A)=W(A')$.
  The same argument applies to $B$ and $B'$.
\end{proof}

\begin{lem}[Shift equivalence over some subring]
  \label{lem:G-SE-to-H-SE}
  Let $G$ be a finite group and
  let $H$ be a normal subgroup of $G$. Let $A$, $B$ be irreducible
  square matrices over $\bbZP[H]$.
  If $A$ and $B$ are shift equivalent over
  $\bbZP[G]$, there is $g \in G$ such that $A$ and $gBg^{-1}$ are shift equivalent over $\bbZP[H]$.
\end{lem}
\begin{proof}
  Let $R,S$ be a shift equivalence of lag $\ell$ between $A \in
  \bbZP[H]^{n \times n}$ and $B \in \bbZP[H]^{m \times m}$ over
  $\bbZP[G]$.
  In particular $A^\ell=RS$ and $B^\ell=SR$.
   Since $A$ and $B$ are essential, so are $R$ and $S$.
  There is an element $g \in G$ and indices $i,j$ 
  such that $\pi_g(R_{i,j})\neq 0$.

  Now let $s \in \{1,\dots,m\}$ and $t \in \{1,\dots,n\}$ be arbitrary.
  There is $k \in \bbN$ such that $A^k_{t,i} \neq 0$.
  Since $gSA^kRg^{-1}=gB^{k+\ell}g^{-1}$ contains
  only entries over $\bbZP[gHg^{-1}]=\bbZP[H]$,
  the same is true for $g S_{s,t} A^k_{t,i} R_{i,j}g^{-1}$.
  The factor $R_{i,j}g^{-1}$ contains $1_G$ as a summand,
  so we have $gS_{s,t} \in \bbZP[H]$.
  Now either $R_{t,s}=0$ or there is $q\in \{1,\dots,m\}$ such that $g
  S_{q,t} R_{t,s}g^{-1} \neq 0$
  since $S$ is essential. In both cases we have $R_{t,s}g^{-1} \in \bbZP[H]$.
  
  Therefore the pair $(Rg^{-1},g S)$ establishes a shift equivalence over
  $\bbZ_+[H]$
  between $A$ and $g B g^{-1}$.
\end{proof}

We now show that the conditions 
of \Cref{thm:equivariant-flow-equivalence-characterization}
are met when $A$ and $B$ satisfy the conditions
of \Cref{thm:flow-equ}, thus showing that in this case
equivariant shift equivalence implies
equivariant flow equivalence.
To do so, we need another characterization of shift equivalence
coming from the \enquote{positive K-theory} framework.
\begin{thm}[Theorem 6.2 \cite{boyleStrongShiftEquivalence2016}]
  \label{thm:k-theory-characterization-se}
  Let $A,B$ be square matrices over a ring $\mathcal{R}$. Then
  $A$ and $B$ are shift equivalent over $\calR$ iff and only iff
  $(I-tA)_{\infty}$ and $(I-tB)_{\infty}$ are $\GL_\infty(\calR[t])$
  equivalent, i.e. there are $U,V \in \GL_\infty(\calR[t])$
  such that
  \begin{align*}
    U(t) (I-tA)_\infty V(t) = (I-tB)_\infty.
  \end{align*}
\end{thm}

\begin{thm}
  \label{thm:SE-implies-FE}
  Let $G=\bbZ/n\bbZ$ be a cyclic group.
  Let $A,B$ be irreducible square matrices over $\bbZP[G]$.
  If $A$ and $B$ are shift equivalent over $\bbZ_+[G]$,
  then $(X_A,\sigma,\alpha_A)$ and $(X_B,\sigma,\alpha_B)$ are
  $G$-flow equivalent.
\end{thm}
\begin{proof}
  Let $A$ and $B$ be shift equivalent matrices
  as above.
  By \Cref{lem:se-have-same-weight-class}
  they have the same weight class.
  Let $H$ be a group in this weight class.
  Since $G$ is Abelian, $H$ is normal.
  Then by \cite[Proposition 4.4]{boyleEquivariantFlowEquivalence2005}
  there are diagonal matrices $C,D$ with
  diagonal entries in $G$ such that
  $A' := DAD^{-1}$ and $B' :=CBC^{-1}$
  are matrices over $\bbZP[H]$.
  $A$ and $A'$ are strongly shift equivalent over $\bbZ_+[G]$.
  The same holds for $B'$ and $B$.
  Additionally $A, A', B$ and $B'$ are all
  shift equivalent to each other over $\bbZP[G]$.
  
  Therefore $A'$ and $B'$ are also shift equivalent over $\bbZP[H]$
  by \Cref{lem:G-SE-to-H-SE}.
  
  By \Cref{thm:k-theory-characterization-se} we know that
  there are matrices $U,V \in \GL_\infty(\bbZ[H][t])$
  such that
  \begin{align*}
    U(t) (I-tA')_\infty V(t) = (I-tB')_\infty.
  \end{align*}
  We now have to upgrade this $\GL_\infty(\bbZ[H][t])$-equivalence
  to an $\SL_\infty(\bbZ[H][t])$-equivalence.
  Since $G$ and hence also $H$ are Abelian by assumption, both $\bbZ[H]$
  and $\bbZ[H][t]$ are commutative rings.
  Therefore we can take the determinant and, setting
  $u(t) := \det U(t)$, $v(t) := \det V(t)$, we see that $u(t)$ and $v(t)$
  are units in $\bbZ[H][t]$ which satisfy
  \begin{align*}
    u(0)v(0)=1.
  \end{align*}
  Now $\bbZ[H]$ contains no nilpotent elements, which can be seen
  e.g. by applying Wedderburn's theorem to the commutative group algebra
  $\bbC[H]$.
  Therefore the units in $\bbZ[H][t]$ are precisely
  the units in $\bbZ[H]$, as the coefficients of order
  at least one in a unit in a commutative polynomial ring have to be nilpotent,
  see e.g. \cite[Exercise 7.3.33]{dummitAbstractAlgebra2004}.
  Hence $u(t)=u(0)$, $v(t)=v(0)$ and therefore
  $u(t)v(t)=1$. Conjugating our equation above
  by $u(t)$ and setting $t=1$ gives
  \begin{align*}
    u(1)^{-1} U(1) (I-A')_\infty V(1)v(1)^{-1} = (I-B')_\infty.
  \end{align*}
  Both matrices $u(1)^{-1}U(1)$ and $V(1)v(1)^{-1}$ are
  contained in \[\SL_\infty(\bbZ[H]) = \{W \in \GL_\infty(\bbZ[H])
  \setsep \det(W)=1\}.\]
  By a K-theoretic result of Oliver, see \cite[Theorem 14.2 (iii)]{oliverWhiteheadGroupsFinite1988}
  we have $\SL_\infty(\bbZ[H])=\EL_\infty(\bbZ[H])$ for the cyclic group
  $H$.
  Hence $(I-A')_\infty$ and $(I-B')_\infty$ are $\EL_\infty(\bbZ[H])$
  equivalent.
  Therefore \Cref{thm:equivariant-flow-equivalence-characterization}
  \eqref{cond:equ-flow-eq-chara-condition}
  is satisfied and $(X_{\calE(A)},\sigma,\alpha_A)$
  and $(X_{\calE(B)},\sigma,\alpha_B)$
  are $G$-flow equivalent.
\end{proof}

\begin{rem}
  The precise condition
    in Olivier's result for $\SL_\infty(\bbZ[H])=\EL_\infty(\bbZ[H])$
    with $H$ Abelian
    is that either (a) each Sylow subgroup of $H$ has the form
    $\bbZ/p^n \bbZ$ or $\bbZ/p \bbZ \times \bbZ/p^n \bbZ$ or (b) $H \cong (\bbZ/2 \bbZ)^k$.
    So instead of cyclicity  of $G$ one case use this condition on $H$
    in the assumptions of \Cref{thm:SE-implies-FE}.
\end{rem}

\begin{cor}
  \label{cor:flow-equ}
  Let $G=\bbZ/n\bbZ$ be the cyclic group of order $n \in \bbN$.
  Let $X$ be the full $k$-shift for some $k \in \bbN$.
  Let $(X,\sigma,\alpha)$ and $(X,\sigma,\beta)$
  be two free $G$-SFTs. Then there
  is a $G$-equivariant flow equivalence
  between $(X,\sigma,\alpha)$ and $(X,\sigma,\beta)$.
\end{cor}
\begin{proof}
  This is a direct corollary of \Cref{thm:main-corr-classic-version}
  and \Cref{thm:SE-implies-FE}. 
\end{proof}

\section*{Acknowledgments}
The author thanks Scott Schmieding and Tom Meyerovitch for stimulating
discussions about $G$-SFTs.

\printbibliography

\end{document}